\documentclass[11pt,a4paper,reqno]{amsart}
\usepackage[english]{babel}
\usepackage[applemac]{inputenc}
\usepackage[T1]{fontenc}
\usepackage{palatino}
\usepackage{amsmath}
\usepackage{amssymb}
\usepackage{amsthm}
\usepackage{amsfonts}
\usepackage{graphicx}
\usepackage{color}
\usepackage{mathtools}

\usepackage[colorlinks = true, citecolor = black]{hyperref}
\title[Kakeya inequality in the Heisenberg group]{Kakeya maximal inequality in the Heisenberg group}

\author[K. F\"assler]{Katrin F\"assler}
\address{Department of Mathematics and Statistics\\ University of Jyv\"askyl\"a \\ P.O. Box 35 (MaD),
FI-40014 University of Jyv\"askyl\"a, Finland}
\email{katrin.s.fassler@jyu.fi}

\author[A. Pinamonti]{Andrea Pinamonti}
\address{Department of Mathematics\\ University of Trento\\ Via Sommarive, 14 \\ IT-38123 Povo, Italy}
\email{andrea.pinamonti@unitn.it}

\author[P. Wald]{Pietro Wald}
\address{Mathematics Institute\\
University of Warwick\\ Zeeman Building\\
Coventry CV4 7AL, UK} \email{pietro.wald@warwick.ac.uk}

\thanks{K.F.\ was supported by the
Academy of Finland grants 321696, 328846.}
\thanks{A.P.\ was supported by the INdAM- GNAMPA Project 2022 Problemi al bordo e applicazioni geometriche, codice CUP\_E55\-F22\-00\-02\-70\-001. }
\thanks{P.W.\ was supported by EPSRC DTP during the later stage of this work.}
\date{\today}

\newcommand{\R}{\mathbb{R}}

\newcommand{\Z}{\mathbb{Z}}

\newcommand{\calT}{\mathcal{T}}

\newcommand{\calH}{\mathcal{H}}

\newcommand{\calB}{\mathcal{B}}

\newcommand{\spa}{\operatorname{span}}

\newcommand{\card}{\operatorname{card}}
\newcommand{\dist}{\operatorname{dist}}

\newcommand{\Hei}{{\mathbb{H}}^{1}}

\def\Barint_#1{\mathchoice
          {\mathop{\vrule width 6pt height 3 pt depth -2.5pt
                  \kern -8pt \intop}\nolimits_{#1}}%
          {\mathop{\vrule width 5pt height 3 pt depth -2.6pt
                  \kern -6pt \intop}\nolimits_{#1}}%
          {\mathop{\vrule width 5pt height 3 pt depth -2.6pt
                  \kern -6pt \intop}\nolimits_{#1}}%
          {\mathop{\vrule width 5pt height 3 pt depth -2.6pt
                  \kern -6pt \intop}\nolimits_{#1}}}

\DeclarePairedDelimiter\abs{\lvert}{\rvert}
\DeclarePairedDelimiter\norm{\lVert}{\rVert}

\numberwithin{equation}{section}

\theoremstyle{plain}
\newtheorem{thm}[equation]{Theorem}

\newtheorem{lemma}[equation]{Lemma}

\newtheorem{cor}[equation]{Corollary}
\newtheorem{proposition}[equation]{Proposition}

\theoremstyle{definition}

\newtheorem{definition}[equation]{Definition}

\theoremstyle{remark}
\newtheorem{remark}[equation]{Remark}

\usepackage{hyperref}
\hypersetup{%
  colorlinks = true,
  linkcolor  = red
}

\usepackage{thmtools}
\usepackage{thm-restate}

\begin{document}

\begin{abstract}
We define the Heisenberg Kakeya maximal functions $M_{\delta}f$,
$0<\delta<1$, by averaging over $\delta$-neighborhoods of
horizontal unit line segments in the Heisenberg group
$\mathbb{H}^1$ equipped with the Kor\'{a}nyi distance
$d_{\mathbb{H}}$. We show that
\begin{displaymath}
\|M_{\delta}f\|_{L^3(S^1)}\leq
C(\varepsilon)\delta^{-1/3-\varepsilon}\|f\|_{L^3(\mathbb{H}^1)},\quad
f\in L^3(\mathbb{H}^1),
\end{displaymath}
for all $\varepsilon>0$. The proof is based on a recent variant,
due to Pramanik, Yang, and Zahl, of Wolff's circular maximal
function theorem for a class of planar curves related to Sogge's
cinematic curvature condition. As an application of our Kakeya
maximal inequality, we recover the sharp lower bound for the
Hausdorff dimension of Heisenberg Kakeya sets of horizontal unit
line segments in $(\mathbb{H}^1,d_{\mathbb{H}})$, first proven by
Liu.
\end{abstract}

\maketitle

\section{Introduction}\label{s:Intro}
\noindent This paper concerns the Heisenberg group
$\mathbb{H}^1=(\mathbb{R}^3,\cdot)$ with the product
\begin{displaymath}
(x_1,x_2,x_3) \cdot (x_1',x_2',x_3') =
\Big(x_1+x_1',x_2+x_2',x_3+x_3'+\tfrac{1}{2}[x_1 x_2'-x_2
x_1']\Big)
\end{displaymath}
and the \emph{Kor\'{a}nyi metric}
$d_{\mathbb{H}}(x,y)=\|y^{-1}\cdot x\|_{\mathbb{H}}$, where $
\|x\|_{\mathbb{H}}=\left((x_1^2+x_2^2)^2+ 16 x_3^2\right)^{1/4}$.
We introduce the \emph{Heisenberg Kakeya maximal function}
$M_{\delta}f:S^1 \to [0,\infty]$,
\begin{equation}\label{eq:MaxOpDef}
M_{\delta}f(e)= \sup_{y\in \mathbb{H}^1}
\frac{1}{|T_{\delta}(y,e)|} \int_{T_{\delta}(y,e)} |f|,\quad e\in
S^1.
\end{equation}
Here $T_{\delta}(y,e)$ is the \emph{Heisenberg $\delta$-tube} of
length $1$ at $y$ in direction $e$, as defined in Definition
\ref{d:HeisTube}, and integration is with respect to Lebesgue
measure on $\mathbb{R}^3$. The coaxial lines of such tubes are
\emph{horizontal lines}  in the sense of Heisenberg geometry, or,
equivalently, lines in $\mathcal{L}_{SL(2)}$ in the terminology of
\cite{2022arXiv221009581W,2022arXiv221009955F,2022arXiv221105194K}.
We use a $\delta$-incidence result for arcs of parabolas in
$\mathbb{R}^2$ to deduce information about Heisenberg Kakeya
maximal functions. More precisely, we apply a special case of
Pramanik, Yang, and Zahl's recent generalization
\cite{2022arXiv220702259P} of Wolff's circular maximal function
bound \cite{MR1473067} to prove the following result.

\begin{restatable}[Heisenberg Kakeya maximal function inequality]{thm}{MainIntro}
\label{t:MainIntro} For all $\varepsilon>0$, $\delta \in (0,1)$,
\begin{displaymath}
\|M_{\delta}f\|_{L^3(S^1)} \lesssim_{\varepsilon}
\delta^{-1/3-\varepsilon}\,\|f\|_{L^3(\mathbb{H}^1)},\quad f\in
L^3(\mathbb{H}^1).
\end{displaymath}
\end{restatable}
The inequality of Theorem \ref{t:MainIntro} is the the best
possible in the sense discussed in Remark \ref{rmk:sharpness}; see
also Remark \ref{rmk:interpolation}.

\subsection{Implications of the main result and related work}
Our definition of Heisenberg Kakeya maximal functions is inspired
by the \emph{Euclidean} Kakeya maximal functions
$\mathcal{K}_{\delta}f$ introduced by Bourgain \cite{MR1097257}.
The prominent \emph{Kakeya maximal conjecture} states that
\begin{displaymath}
\|\mathcal{K}_{\delta}f\|_{L^n(S^{n-1})}\lesssim_{n,\varepsilon}
\delta^{-\varepsilon} \,\|f\|_{L^n(\mathbb{R}^n)}\quad\text{for
all }\varepsilon>0,\, 0<\delta<1.
\end{displaymath}
This is currently known only for $n=2$, where it was first proven
by C\'{o}rdoba \cite{MR447949}.
For a survey of the important developments related to Euclidean
Kakeya maximal inequalities, we refer the reader to, for instance,
\cite{MR1660476,MR2003254,MR3617376}. As far as we know, Theorem 1
does not entail progress related to the Kakeya maximal conjecture
in $\mathbb{R}^3$. Its setting differs from that of standard
Euclidean Kakeya inequalities in two crucial aspects: The
``$\delta$-tubes''  used in the definition of the Heisenberg
Kakeya maximal function
\begin{enumerate}
\item[(i)] are defined using the \emph{Kor\'{a}nyi distance}
 and hence have volume comparable to $\delta^3$,
\item[(ii)] arise as \emph{Heisenberg left translates} of a
one-dimensional family of Heisenberg tubes $\{T_{\delta}(0,e):\,
e\in S^1\}$ pointing in the \emph{horizontal directions}
$\{e\}\times \{0\}$, $e\in S^1$.
\end{enumerate}
Our definition is tailored to the geometry of $\mathbb{H}^1$. As a
corollary of Theorem \ref{t:MainIntro}, we recover a result,
originally due to Liu \cite{MR4439466}, which states that
\emph{Heisenberg Kakeya sets} of horizontal unit line segments in
$\mathbb{H}^1$ (in the sense of Definition \ref{d:HeisKakeya})
have Hausdorff dimension at least $3$ with respect to
$d_{\mathbb{H}}$. This bound is sharp, as evidenced by the
horizontal plane $\{x_3=0\}$. In contrast, a union of horizontal
line segments whose directions range in a positive measure subset
of $S^2$ in the classical sense must necessarily have Hausdorff
dimension $4$ with respect to $d_{\mathbb{H}}$. This follows from
the full dimensionality of such unions with respect to Euclidean
metric, proven recently by Orponen and the first author in
\cite{2022arXiv221009955F}, and with a different technique by
Katz, Wu, and Zahl in \cite{2022arXiv221105194K}. The proof in
\cite{2022arXiv221009955F} used a Marstrand-type theorem for a
restricted family of orthogonal projections onto \emph{planes} in
$\mathbb{R}^3$, while the  results in the present paper and in
\cite{MR4439466} are conceptually related to a projection theorem
for \emph{lines}, which seems well suited when the Heisenberg
metric is used,  see Section \ref{ss:ouline}.

\medskip

\noindent We briefly discuss the relation with other results in
the literature. The version of Wolff's circular maximal function
by Pramanik, Yang, and Zahl \cite{2022arXiv220702259P} was
recently also applied in \cite{2022arXiv221009581W} to prove a
special case of the Kakeya conjecture in $\mathbb{R}^3$, the
\emph{sticky Kakeya set conjecture}, and in
\cite{2022arXiv221105194K} for the \emph{$SL(2)$ Kakeya
conjecture}. Since in our case, the coaxial lines of the relevant
tubes  are  horizontal lines, we do not need the full strength of
\cite{2022arXiv220702259P}, but only a special case for quadratic
functions. On the other hand, the proof of Theorem
\ref{t:MainIntro} necessitates considerations related to the
geometry of $(\mathbb{H}^1,d_{\mathbb{H}})$, as mentioned in  (i)
-(ii).

Certain lower bounds for the Hausdorff dimension with respect to
$d_{\mathbb{H}}$ of \emph{standard} (not Heisenberg) Kakeya sets
in $\mathbb{R}^{2n+1}$ were deduced by Venieri \cite{MR3290381}
from known bounds  for the
 \emph{standard} Kakeya maximal functions.  On the other hand,
 Euclidean Kakeya problems with a restricted set of directions  were studied in \cite{MR1922668,MR3676055,2022arXiv220315731K,2022arXiv220401408G}, but Heisenberg Kakeya sets lie outside this scope. The reason is that
the left translate $y\cdot I_e$ of a segment $I_e$ with respect to
the Heisenberg product need not be parallel to $I_e$ in the
Euclidean space $\mathbb{R}^3$. The Heisenberg Kakeya maximal
functions introduced in the present paper are closer in spirit to
the Nikodym maximal functions defined by Kim in
\cite{MR2541278,MR2956810}, but he works with different types of
tubes and  his results do not seem to have direct implications
regarding the Heisenberg Kakeya inequality studied here. Finally,
we mention that Venieri formulated in \cite{MR3676055} an
axiomatic framework for deriving Kakeya-type inequalities and
Hausdorff dimension lower bounds for Kakeya-type sets using ideas
by Bourgain and Wolff, which however does not seem to yield the
inequality in Theorem \ref{t:MainIntro}. Nonetheless, the
arguments we use to derive the sharp Hausdorff dimension bound for
Heisenberg Kakeya sets fit well in this axiomatic
 framework,
cf., e.g., the proof of \cite[Theorem 4.1]{MR3676055}.

\subsection{Outline of the proofs}\label{ss:ouline} Theorem
\ref{t:MainIntro} can be obtained by duality from the following
discretized Heisenberg Kakeya inequality.
\begin{restatable}[Kakeya inequality for Heisenberg tubes]{thm}{LinKakeyaIneqIntro}
\label{t:LinKakeyaIneqIntro} Let $\delta\in (0,1)$ and assume that
$\calT$ is a family of Heisenberg $\delta$-tubes pointing in
$\delta^2$-separated directions of $S^1$. Then, for all
$\varepsilon>0$,
\begin{equation}\label{eq:t:LinKakeyaIneq}
    \int \left(\sum_{T\in\calT}\chi_{T}\right)^{3/2}\lesssim_\varepsilon
    \delta^{3-\varepsilon}\mathrm{card}(\calT).
\end{equation}
\end{restatable}
The fact that Theorem \ref{t:LinKakeyaIneqIntro} can be applied to
(maximal) families of tubes pointing in $\delta^2$-separated
directions -- and not only to sparser collections of directionally
$\delta$-separated tubes -- is crucial in deriving Theorem
\ref{t:MainIntro} from Theorem \ref{t:LinKakeyaIneqIntro} with the
help of Lemma \ref{l:KakeyaMax_small_angle_diff}.  This is a
particular feature of the Heisenberg group, see Remark
\ref{rmk:exponent_mismatch}.

\medskip

To understand why Theorem \ref{t:LinKakeyaIneqIntro} holds for
Heisenberg $\delta$-tubes pointing in $\delta^2$-separated
directions, we reduce its proof to a Kakeya-type inequality for
neighborhoods of arcs of parabolas in $\mathbb{R}^2$, see
Corollary \ref{cor:LinKakeyaParabolas}. A crucial tool for  this
reduction is the \emph{vertical Heisenberg projection}
\begin{equation}\label{eq:proj}
\pi_{\mathbb{W}}:\mathbb{H}^1 \to \mathbb{W},\quad
\pi_{\mathbb{W}}\left((0,x_2,x_3)\cdot (x_1,0,0)\right)=
(0,x_2,x_3),
\end{equation}
where ``$\cdot$'' denotes the group product. This bears
similarities with the use of the \emph{twisted projections} in
\cite{2022arXiv221009581W}; see \cite[Definition
6.1]{2022arXiv221009581W}  for $f(z)=z/2$ and its relation with
\cite[(7.4)]{2022arXiv221009581W}. Lemma \ref{l:tube_parab} shows
that Heisenberg $\delta$-tubes in $\delta^2$-separated directions
are mapped by $\pi_{\mathbb{W}}$ into $\sim
\delta^2$-neighborhoods of graphs of quadratic polynomials with
$\sim \delta^2$-separated leading coefficients, at least under
suitable assumptions on the position and direction of the tubes.
This allows to deduce Theorem \ref{t:LinKakeyaIneqIntro} from
Corollary \ref{cor:LinKakeyaParabolas} applied at scale
``$\delta^2$''. Finally, Corollary \ref{cor:LinKakeyaParabolas} is
obtained as a simple special case of \cite[Theorem
1.7]{2022arXiv220702259P}.

Liu's earlier work on Heisenberg Kakeya sets \cite{MR4439466} was
based on a Marstrand-type theorem for the almost sure Hausdorff
dimension of orthogonal projections from $\mathbb{R}^3$ onto
one-dimensional subspaces foliating the surface of a cone in
$\mathbb{R}^3$. This projection theorem was first proven by
K\"aenm\"aki, Orponen, and Venieri in
\cite{https://doi.org/10.48550/arxiv.1708.04859}, and recently
sharpened and extended to a larger class of families of
one-dimensional subspaces by Pramanik, Yang, and Zahl using the
Kakeya-type inequality in \cite[Theorem 1.7]{2022arXiv220702259P}
mentioned above. The projection theorem was independently and
simultaneously extended in \cite{2022arXiv220915152G} with
different methods, but here we discuss especially the approaches
in
\cite{https://doi.org/10.48550/arxiv.1708.04859,2022arXiv220702259P}
since they were both inspired by Wolff's work on the circular
Kakeya problem. In particular, Liu's proof of the dimension bound
for Kakeya sets is ultimately based on Wolff's results for
$\delta$-annuli in $\mathbb{R}^2$, albeit indirectly via the
projection theorem. By using the recent generalization
\cite{2022arXiv220702259P} of Wolff's work, which directly applies
to arcs of parabolas in $\mathbb{R}^2$, we are able to make use of
planar incidence geometry in a more direct way and obtain a
stronger conclusion in the form of a Heisenberg Kakeya maximal
inequality.

\subsection{Structure of the paper} Section \ref{ss:tubes} contains
preliminaries on Heisenberg $\delta$-tubes and cinematic
functions.  In Section \ref{ss:LinKakeya}, we prove the main
results of the paper, the Heisenberg Kakeya inequalities in
Theorems \ref{t:MainIntro} and \ref{t:LinKakeyaIneqIntro}.
Finally, in Section \ref{ss:Dim}, we apply  Theorem
\ref{t:MainIntro} to give a new proof for Liu's theorem on the
dimension of Heisenberg Kakeya sets.

\subsection{Notation}  If $f,g\geq 0$, the notation $f\lesssim g$
denotes the existence of a positive  constant $C$ such that $f\leq
C g$. The notation $f\lesssim_{\kappa} g$ means that $C$ may
depend on a parameter ``$\kappa$''. Finally, $f\sim g$ is an
abbreviation of $f\lesssim g \lesssim f$. We denote by $|E|$ the
$d$-dimensional Lebesgue measure of a measurable set $E\subset
\mathbb{R}^d$. The $s$-dimensional Hausdorff measure on
$\mathbb{H}^1$ with respect to $d_{\mathbb{H}}$ is denoted by
$\mathcal{H}^s_{\mathbb{H}}$, or  $\mathcal{H}^s$ if the metric is
clear from the context.  Kor\'{a}nyi balls are denoted by
$B_{\mathbb{H}}(x,r)$ or $B(x,r)$, and Euclidean balls by
$B_E(x,r)$.

\subsection*{Acknowledgements} Part of this work was
done while K.F.\ and P.W.\  visited the Hausdorff Research
Institute for Mathematics, Bonn, during the trimester
\emph{Interactions between Geometric measure theory, Singular
integrals, and PDE}. The hospitality of the institute is
gratefully acknowledged. Part of this work was done while P.W.\
was a master student at the University of Trento. We also  wish to
thank Tuomas Orponen and Joshua Zahl for inspiring discussions in
Bonn that gave  an important impetus for the present paper.

\section{Preliminaries on Heisenberg tubes and cinematic functions}\label{ss:tubes}

\subsection{Heisenberg tubes} We consider the Heisenberg group
$\mathbb{H}^1 =(\mathbb{R}^3,\cdot)$ with the left-invariant
Kor\'{a}nyi metric as defined in Section \ref{s:Intro}. For a
thorough introduction to this space, we refer the reader to
\cite{CDPT}. To define Heisenberg Kakeya maximal functions, we
need \emph{Heisenberg tubes}, whose definition and elementary
properties we now discuss.

\begin{definition}[Heisenberg tubes]\label{d:HeisTube} Let $0<\delta<1$. Given $y\in \mathbb{H}^1$ and $e\in S^1$, the \emph{Heisenberg
$\delta$-tube} $T_{\delta}(y,e)$ is the $\delta$-neighborhood (in
the metric $d_{\mathbb{H}}$) of the horizontal line segment
 $y\cdot I_e$, where  $I_e = \{(se,0): s\in [-1/2,1/2]\}$. We also say
 that $T_{\delta}(y,e)$ \emph{points in direction~$e$}.
\end{definition}

It is often convenient to write Heisenberg $\delta$-tubes in the
following more explicit form.

\begin{lemma}\label{l:Heis_tube}
For every $y\in \mathbb{H}^1$, $e\in S^1$, and $\delta>0$,
\begin{displaymath}
T_{\delta}(y,e)= \{y\cdot (se,0)\cdot B(0,\delta):\, s\in
[-1/2,1/2]\},
\end{displaymath}
where $B(0,\delta)=\{x\in \mathbb{H}^1:\, \|x\|_{\mathbb{H}}\leq
\delta\}$.
\end{lemma}

\begin{proof}
To prove the inclusion ``$\subseteq$'', we consider an arbitrary
point $x\in T_{\delta}(y,e)$. Then there exists $s\in [-1/2,1/2]$
such that
\begin{displaymath}
\|[y\cdot (se,0)]^{-1}\cdot x \|_{\mathbb{H}} =
d_{\mathbb{H}}(x,y\cdot (se,0) )\leq \delta.
\end{displaymath}
Writing $x = y\cdot (se,0) \cdot [y\cdot (se,0)]^{-1}\cdot x$ then
proves the desired inclusion. For the reverse inclusion, it
suffices to observe for all $z\in B(0,\delta)$ that
\begin{displaymath}
d_{\mathbb{H}}\left(y\cdot (se,0)\cdot z,y \cdot
(se,0)\right)=\|z\|_{\mathbb{H}} \leq \delta.\qedhere
\end{displaymath}
\end{proof}

The Heisenberg Kakeya maximal function $M_{\delta}f$ is defined by
taking averages over Heisenberg $\delta$-tubes with respect to the
$3$-dimensional Lebesgue measure on the underlying $\mathbb{R}^3$.
We recall that this measure is invariant under translations with
respect to the Heisenberg group product, and it agrees up to a
positive and finite multiplicative factor with the $4$-dimensional
Hausdorff measure $\mathcal{H}^4$ with respect to
$d_{\mathbb{H}}$. The space $(\mathbb{H}^1,d_{\mathbb{H}})$ is
topologically $3$-dimensional, but Ahlfors $4$-regular. While a
Euclidean $\delta$-tube of length $1$ in $\mathbb{R}^3$ has volume
$\sim \delta^2$, the volume of Heisenberg $\delta$-tubes is $\sim
\delta^3$.

\begin{lemma}[Volume of a Heisenberg tube]\label{l:vol} For each
$\delta\in (0,1)$, Heisenberg $\delta$-tubes have volume
\begin{displaymath}
|T_{\delta}(y,e)|\sim \delta^3,\quad y\in \mathbb{H}^1,\, e\in
S^1.
\end{displaymath}
\end{lemma}

\begin{proof}
Given a tube $T=T_{\delta}(y,e)$ as in the statement of the lemma,
let $\{x_1,\ldots,x_N\}$ be a maximal $\delta$-separated subset of
its core segment $y \cdot I_e$. Then $N\sim \delta^{-1}$, because
$(y\cdot I_e,d_{\mathbb{H}})$ is isometric to $([0,1],|\cdot|)$.
Now the open balls $B_{\mathbb{H}}(x_i,\delta/2)$, $i=1,\ldots,N$,
are pairwise disjoint, have volume comparable to $\delta^4$ and
are all contained in $T$, so that $|T|\gtrsim \delta^3$. On the
other hand, the union of the balls $B_{\mathbb{H}}(x_i,2\delta)$,
$i=1,\ldots,N$, covers  $T$, hence $|T|\lesssim \delta^3$.
\end{proof}

\subsection{From Heisenberg tubes to neighborhoods of
parabolas}\label{ss:FromTubeToPar} We now discuss connections
between Heisenberg $\delta$-tubes in $\mathbb{H}^1$ and Euclidean
$\delta^2$-neighborhoods of arcs of parabolas in $\mathbb{R}^2$.
For $\delta=0$, i.e., for horizontal line segments and parabolic
arcs, such connections were used by Liu \cite{MR4439466} to prove
the dimension bound for Heisenberg Kakeya sets. The relevant
parabolas take a particularly simple form if Heisenberg tubes are
described by parameters $a,b,c$ and $y_2$ as in the following
lemma.


\begin{lemma}\label{l:TubeParam} If
$e =(\cos \varphi,\sin \varphi)\in S^1$ for $\varphi\in (0,\pi)$
and  $y=(y_1,y_2,y_3)$, then
\begin{equation}\label{eq:TubeParam}
T_{\delta}(y,e)=  \left\{(b,0,c)\cdot (as,s,0)\cdot B(0,\delta):\,
s\in [s_{-},s_{+}]\right\},
\end{equation}
where $s_{\pm}= y_2 \pm \frac{1}{2 \sqrt{1+a^2}}$ and
\begin{equation}\label{eq:(a,b,c)} a= \frac{\cos
\varphi}{\sin \varphi},\quad b = y_1 -\left(\tfrac{\cos
\varphi}{\sin \varphi}\right) y_2,\quad c = y_3-\tfrac{1}{2}y_1
y_2 + \tfrac{1}{2} \left(\tfrac{\cos \varphi}{\sin \varphi}\right)
y_2^2.
\end{equation}
\end{lemma}

\begin{proof}   By Lemma
\eqref{l:Heis_tube}, it suffices to observe that the horizontal
core segment of $T_{\delta}(y,e)$ can be parameterized as follows
\begin{equation}\label{eq:segParam}
y\cdot I_e =  \{(y_1,y_2,y_3)\cdot (as,s,0):\, s\in
[s_{-},s_{+}]\},
\end{equation}
which was already shown in \cite[Lemma 2.1]{MR4439466} (up to an
obvious change in the roles of the first and second coordinate
axis). Then \eqref{eq:TubeParam} follows since
\begin{align*}
(y_1,y_2,y_3)\cdot (as,s,0)&=
(y_1+as,y_2+s,y_3+\tfrac{1}{2}[y_1-ay_2]s)\\
&= ([y_1-a y_2] + a(y_2+s),y_2+s, y_3 -\tfrac{1}{2} [y_1-ay_2]y_2
+\tfrac{1}{2}[y_1-ay_2](y_2+s))\\
&= (b + a(y_2+s),y_2+s,c+ \tfrac{1}{2}b(y_2+s)).
\end{align*}
\end{proof}


The next result relates Heisenberg tubes to Euclidean
neighborhoods of arcs of parabolas in $\mathbb{R}^2$ via the
vertical Heisenberg projection in \eqref{eq:proj} onto the plane
$\mathbb{W}=\{x_1=0\}$. Explicitly, in coordinates,
\begin{displaymath}
\pi_{\mathbb{W}}:\mathbb{H}^1 \to \mathbb{W},\quad
\pi_{\mathbb{W}}(x_1,x_2,x_3)=\left(0,x_2,x_3+\tfrac{1}{2}x_1
x_2\right).
\end{displaymath}
A related statement appeared in  \cite[Lemma 4.5]{MR3047423}, but
our setting is a little different. For $(a,b,c)\in \mathbb{R}^3$,
we let $\gamma_{(a,b,c)}$ be the parabola in $\mathbb{R}^2$
parameterized by
\begin{displaymath}
\gamma_{(a,b,c)}(s):=(s,\tfrac{a}{2}s^2+bs +c),\quad s\in
\mathbb{R}.
\end{displaymath}
By $[\Gamma]^r$ we denote the Euclidean $r$-neighborhood of a set
$\Gamma\subset \mathbb{R}^2$.

\begin{lemma}[Projections of segments and tubes]\label{l:tube_parab} Let $0<\delta<1$.
For all $e =(\cos \varphi,\sin \varphi)$ with $\varphi\in (0,\pi)$
and $y=(y_1,y_2,y_3)\in \mathbb{R}^3$, we have that
\begin{displaymath}
\pi_{\mathbb{W}}\left(y\cdot
I_{e}\right)=\gamma_{(a,b,c)}\big([s_{-},s_{+}]\big) \quad
\text{and} \quad
\pi_{\mathbb{W}}\left(T_{\delta}(y,e)\right)\subset
\left[\gamma_{(a,b,c)}\big(\big[s_{-}-\delta,s_{+}+\delta\big]\big)\right]^{r},
\end{displaymath}
where $s_{\pm}$ and $(a,b,c)$ are as in Lemma \ref{l:TubeParam},
and $r\sim (1+|a|)\delta^2$.

\end{lemma}


\begin{proof}

First, to prove the claim concerning the projected segment, we
recall from the proof of
 Lemma \ref{l:TubeParam} that
\begin{align*}
 y\cdot I_{e} &= \{(as+b,s,c+\tfrac{b}{2}s):\, s\in
 \left[s_{-},s_{+}\right]\}.
 \end{align*}
 It follows immediately that
\begin{displaymath}
 \pi_{\mathbb{W}}\left(y\cdot I_{e_a}\right) = \{(s,\tfrac{a}{2}s^2+bs+c):\, s\in
 \big[s_{-},s_{+}\big]\} =
 \gamma_{(a,b,c)}\big(\big[s_{-},s_{+}\big]\big),
 \end{displaymath}
 where we have identified
 $\{0\}\times \mathbb{R}^2$ with $\mathbb{R}^2$ in the obvious
 way. Next, to prove the claim about the projected tube, we fix $\delta\in (0,1)$ and
consider an arbitrary point
\begin{displaymath}
(b,0,c)\cdot (as,s,0)\cdot z \in y\cdot I_{e} \cdot
B(0,\delta)\overset{\text{Lem. \ref{l:Heis_tube}}}{=}
T_{\delta}(y,e_a),
\end{displaymath}
where $s\in [s_{-},s_{+}]$ and $z\in B(0,\delta)$. By similar
computations as before
\begin{align*}
\pi_{\mathbb{W}}&\left((b,0,c)\cdot (as,s,0)\cdot z \right)\\& =
\left(s+z_2,c +\tfrac{b}{2}s+ z_3+
\tfrac{1}{2}[b+as]z_2-\tfrac{1}{2}s z_1+
\tfrac{1}{2}[b+as+z_1][s+z_2]\right)\\
&= \left(s+z_2,\tfrac{a}{2}s^2 + bs +c
+[b+as]z_2+z_3+\tfrac{1}{2}z_1z_2\right).
\end{align*}
On the other hand,
\begin{align*}
\gamma_{(a,b,c)}(s+z_2)&=\left(s+z_2,\tfrac{a}{2}[s+z_2]^2 +
b[s+z_2]+c\right)\\
&=(s+z_2,\tfrac{a}{2}s^2 + bs +c+[b+as]z_2+\tfrac{a}{2}z_2^2).
\end{align*}
Since $\|z\|_{\mathbb{H}}\leq \delta$, it follows that
\begin{displaymath}
\left|\pi_{\mathbb{W}}\left((b,0,c)\cdot (as,s,0)\cdot z
\right)-\gamma_{(a,b,c)}(s+z_2)\right| = \left|z_3+\tfrac{z_1
z_2}{2}-\tfrac{a}{2} z_2^2\right|\lesssim (1+|a|)\delta^2.
\end{displaymath}
This shows that $\pi_{\mathbb{W}}\left(T_{\delta}(y,e)\right)$ is
contained in the Euclidean $r$-neighborhood of the parabola
$\gamma_{(a,b,c)}$ for some $r\sim (1+|a|)\delta^2$. To conclude
the proof, it suffices now to observe that $s+z_2 \in
[s_{-}-\delta,s_{+}+\delta]$.
\end{proof}

Let $\mathbb{W}$ be the vertical plane $\{x_1=0\}$ and
$\mathbb{L}$ the $x_1$-axis. Together with Lemma
\ref{l:tube_parab}, the next result will allow us to reduce the
Heisenberg Kakeya inequality in Theorem \ref{t:LinKakeyaIneqIntro}
to a corresponding inequality for neighborhoods of parabolic arcs
in $\mathbb{R}^2$ by means of the Fubini-type formula
\begin{equation}\label{eq:Fubini}
\int_{\mathbb{H}^1} h(x) dx= \int_{\mathbb{W}}\int_{\mathbb{L}}
h((0,y,t)\cdot (x,0,0))\,d(y,t)\,dx.
\end{equation}
which holds for nonnegative measurable functions $h$ since
$\Phi(x,y,t)=(0,y,t)\cdot(x,0,0)=(x,y,t-\tfrac{1}{2}xy)$ has
Jacobi determinant $1$. The use of \eqref{eq:Fubini} in our
context requires us to understand how the fibres
$\pi_{\mathbb{W}}^{-1}(w) = w\cdot \mathbb{L}$, for $w\in
\mathbb{W}$, intersect a Heisenberg $\delta$-tube.


\begin{lemma}\label{l:tube_cap_line_lesssim_delta}
Let $\mathbb{W}$ be the vertical plane $\{x_1=0\}$ and
$\mathbb{L}$ the $x_1$-axis.
Then
\begin{equation*}
\calH^1\left(T_\delta(y,e)\cap
\pi_\mathbb{W}^{-1}(w)\right)\lesssim
\delta
\end{equation*}
for all $y\in \Hei$, $w\in\mathbb{W}$,
$\delta\in (0,1)$ and $e\in S^1$ making angle at most $\pi/4$ with
the $x_2$-axis.
\end{lemma}

\begin{remark} Heisenberg projections can be defined for
arbitrary vertical planes in $\mathbb{H}^1$, and Lemma
\ref{l:tube_cap_line_lesssim_delta} holds in this generality with
obvious modifications. Indeed, for each $O\in SO(2)$, the map
$R_O:(x',x_3)\in\R^2\times\R\mapsto (Ox',x_3)$ is a
$(d_{\mathbb{H}},d_{\mathbb{H}})$-isometry, and $R_O \circ
\pi_{\mathbb{W}} = \pi_{R_O(\mathbb{W})}\circ R_O$. Thus, Lemma
\ref{l:tube_cap_line_lesssim_delta} holds more generally for
$\mathbb{L}\subset \R^{2}\times\{0\}$ a $1$-dimensional subspace
and $\mathbb{W}=\mathbb{L}^\perp$ the orthogonal complement in
$\R^3$ of $\mathbb{L}$.
\end{remark}

\begin{proof}[Proof of Lemma
\ref{l:tube_cap_line_lesssim_delta}] We can assume w.l.o.g.\
$y\in\mathbb{W}$; indeed $ \calH^1\left(T_\delta(y,e)\cap
\pi_\mathbb{W}^{-1}(w)\right)= \calH^1\left(T_\delta(w^{-1}\cdot
y,e)\cap \mathbb{L}\right). $ Hence we will show for an arbitrary
$\delta$-tube $T=T_{\delta}((y_1,y_2,y_3),e)$ pointing in
direction $e\in S^1$ that
\begin{equation}\label{eq:goalTubeLine}
\calH^1\left(\{s\in \mathbb{R}:\, (s,0,0)\in T\}\right) \lesssim
\delta.
\end{equation}
We employ the orthogonal projection
$P:\R^3\to\R^2\times\{0\}\equiv\R^2$ to reduce the verification of
\eqref{eq:goalTubeLine} to  planar Euclidean geometry. Clearly, if
$(s,0,0)\in T$, then $(s,0)\in P(T)$. Moreover, since
$P:(\mathbb{H}^1,d_{\mathbb{H}})\to (\mathbb{R}^2,|\cdot|)$ is
$1$-Lipschitz, the projection $P(T)$ is contained in the infinite
strip
\begin{displaymath}
S:=S_{\delta}^{\mathbb{R}^2}((y_1,y_2),e):=
\left[(y_1,y_2)+\mathrm{span}(e)\right]^{\delta},
\end{displaymath}
that is, in the Euclidean $\delta$-neighborhood of the line
$P(y)+\mathrm{span}(e)$. Hence,
\begin{displaymath}
\calH^1\left(\{s\in \mathbb{R}:\, (s,0,0)\in T\}\right)  \leq
\calH^1\left(\{s\in \mathbb{R}:\, (s,0)\in S\}\right).
\end{displaymath}
Since the strip $S$ has width $2\delta$ and points in direction
$e$, trigonometry shows that the the $x_1$-axis intersects $S$ in
an interval of length $2\delta/|\langle e, e_2\rangle|$. By the
assumption on the angle between $e$ and the $x_2$-axis
$\mathrm{span}(e_2)$, we know that $|\langle e, e_2\rangle|\gtrsim
1$, and \eqref{eq:goalTubeLine} follows.
\end{proof}


\subsection{ Kakeya inequality for parabolas} In this section, we
recall a special case of a recent Kakeya-type inequality by
Pramanik, Yang, and Zahl \cite{2022arXiv220702259P} that will be
important for our application to the arcs of parabolas that arise
in Lemma \ref{l:tube_parab}.

\begin{remark}\label{r:EarlierZahl}
The scope of the Kakeya-type inequality in \cite[Theorem
1.7]{2022arXiv220702259P} is broader than what is required for our
application; it is formulated for a class of $C^2$
\emph{cinematic functions}. This condition is
 related to Sogge's \emph{cinematic curvature condition}, which
 was used by Kolasa and Wolff \cite{MR1722768}. We could likely also have employed earlier
work by Zahl \cite{MR2990134,MR3231483}, which holds under the
cinematic curvature assumption instead of \cite[Theorem
1.7]{2022arXiv220702259P}. However, the main results in
\cite{MR2990134,MR3231483} are not directly applicable in our
setting
because we are dealing with polynomials with a ``dimensionality''
or ``non-concentration'' condition for
the coefficients of the quadratic term 
which would force us to make adaptations similar to the ones made
in \cite[Lemma B.2]{lindenstrauss2021polynomial}.
\end{remark}

The arcs of the parabolas $\gamma_{(a,b,c)}$ in Lemma
\ref{l:tube_parab} are instances of the kind of curves studied in
\cite{2022arXiv220702259P}. To see this, for $r>0$ and
$z=(a,b,c)\in \mathbb{R}^3$, we define
\begin{displaymath}
f_z(s)=\tfrac{a}{2}s^2 + bs +c,\quad s\in [-r,r].
\end{displaymath}
and we let $\gamma\in C^2(\R,\R^3)$, $\gamma(s) = (s^2/2,s,1)$. It
is useful to note that $f_z(s) = \langle \gamma(s),z\rangle$ for
all $s\in\R$ and $z\in\R^3$, and
\begin{equation}\label{eq:NonDegGamma}
\mathrm{span}\{\gamma(s),\dot{\gamma}(s),\ddot{\gamma}(s)\}=\mathbb{R}^3,\quad
s\in [-r,r]. \end{equation} Hence the following result is
applicable in our setting.

\begin{thm}[Pramanik-Yang-Zahl]\label{t:LinKakeyaPYZ}
Let $\varepsilon>0$ and $r>0$. Let $I$ be a compact interval and
let $\gamma:I \to \mathbb{R}^3$ be a $C^2$ curve satisfying
$\mathrm{span}\{\gamma(s),\dot{\gamma}(s),\ddot{\gamma}(s)\}=\mathbb{R}^3$,
$s\in I$. Then there exists $\delta_0>0$, depending only on
$\varepsilon$, $r$, and $\gamma:I \to \mathbb{R}^3$, such that the
following holds for all $0<\delta\leq \delta_0$. Let
$Z_{\delta}\subset [-r,r]^3\subset \mathbb{R}^3$ be a set with the
non-concentration condition
\begin{equation}\label{eq:KatzTao}
    \mathrm{card}(Z_{\delta}\cap B)\leq \delta^{-\varepsilon}(\rho/\delta)
\end{equation}
for all Euclidean balls $B\subset \mathbb{R}^3$ of radius $\rho
\geq \delta$. Then
\begin{equation*}
    \int_{[0,1]^2}\left(\sum_{z\in Z_{\delta}}\chi_{\Gamma_z^\delta}\right)^{3/2}
    \leq\delta\cdot\delta^{-C\varepsilon}\mathrm{card}(Z_{\delta}),
\end{equation*}
where $C>0$ is a constant depending on $\gamma:I \to
\mathbb{R}^3$,
 and $\Gamma^\delta_z$ is the Euclidean $\delta$-neighborhood of the graph $\Gamma_z=\{(s,\langle \gamma(s),z\rangle):\, s\in I\}$ in $\mathbb{R}^2$.
\end{thm}

Theorem  \ref{t:LinKakeyaPYZ} is essentially a special case of
\cite[Proposition 2.1]{2022arXiv220702259P} with parameters
$\alpha=\zeta=1$ and $E=[0,1]^2$ since $[0,1]^2$ is a
$(\delta,1;1)_1 \times (\delta,1;1)_1$ quasi-product in the
terminology of  \cite{2022arXiv220702259P}; see Remark 1 below
Theorem 1.7 therein. We briefly comment on the three points in
which our formulation differs slightly from the statement of
\cite[Proposition 2.1]{2022arXiv220702259P}:

\begin{enumerate}
\item \cite[Proposition 2.1]{2022arXiv220702259P} was formulated
for curves $\gamma$ in $S^2$. However, it is clear from the proof
that the argument works also for curves in $\mathbb{R}^3$
satisfying all the other assumptions. Indeed, the assumption
$\mathrm{span}\{\gamma(s),\dot{\gamma}(s),\ddot{\gamma}(s)\}=\mathbb{R}^3$
ensures that
\begin{equation}\label{eq:NonDeg}
    |\langle\gamma(s),z\rangle|+ |\langle\dot{\gamma}(s),z\rangle|
    + |\langle\ddot{\gamma}(s),z\rangle|\neq 0, \quad z\in \mathbb{R}^3 \setminus \{0\},\, s\in I.
\end{equation}
Then, for all $z,z'\in \mathbb{R}^3$, we have by compactness of
$S^2$ and $I$ that
\begin{displaymath}
|z-z'|\overset{\eqref{eq:NonDeg}}{\lesssim_{\gamma,I}} \min_{s\in
I}\sum_{k=0}^2 |\langle\gamma^{(k)}(s),z-z'\rangle|\leq
\sum_{k=0}^2 \sup_{s\in
I}|\langle\gamma^{(k)}(s),z-z'\rangle|\lesssim_{\|\gamma\|_{C^2(I)}}
|z-z'|.
\end{displaymath}
The remaining proof proceeds exactly as in
\cite{2022arXiv220702259P}. \item \cite[Proposition
2.1]{2022arXiv220702259P} was formulated for $Z_{\delta}\subset
B_E(0,1)$ instead of $Z_{\delta}\subset [-r,r]^3$, but this change
only influences the cinematic constant ``$K$'' of the family
$\mathcal{F}$ in the proof of  \cite[Proposition
2.1]{2022arXiv220702259P}, and hence the constant $\delta_0$.
\item  \cite[Proposition 2.1]{2022arXiv220702259P} was  formulated
with an additional $\delta$-separateness assumption for
$Z_{\delta}$. However, the version we stated above can easily be
reduced to this case since \eqref{eq:KatzTao} ensures that every
$\delta$-ball in $\mathbb{R}^3$ contains at most
$\delta^{-\varepsilon}$ elements of $Z_{\delta}$. Then a standard
coloring argument allows us to write $Z_{\delta}$ as a union of $M
\sim \delta^{-\varepsilon}$ sets
$Z_{\delta,1},\ldots,Z_{\delta,M}$
 such that each $Z_{\delta,i}$ is $\delta$-separated, see for instance \cite[p.101]{MR1800917}. Thus \cite[Proposition 2.1]{2022arXiv220702259P} can be applied to each set $Z_{\delta,i}$ individually, and the above version follows (with a larger constant $C$) by triangle inequality.
\end{enumerate}

\begin{remark}\label{r:Sep_a}
Condition \eqref{eq:KatzTao} is satisfied in our case thanks to a
non-concen{\-}tration condition for the coefficients ``$a$'' of
the corresponding functions $f_z(s)=\tfrac{a}{2}s^2+ bs +c$.
Indeed, assume that $Z_{\delta} \subset [-r,r]^3$ is such that
$|a-a'|\geq c_0\delta$ for all distinct $z=(a,b,c)$ and
$z'=(a',b',c')\in Z_{\delta}$ and a universal constant $c_0>0$.
This implies that $(a,b,c)\in Z_{\delta}\mapsto a\in A:=\{a\in
\mathbb{R}:\, (a,b,c)\in Z_{\delta}\text{ for some }b,c\in
\mathbb{R}\}$ is bijective and $A$ is a $c_0\delta$-separated set
in $\mathbb{R}$. Now if $B=B_E(z_0,\rho)\subset \mathbb{R}^3$ is a
Euclidean ball with $\rho\geq \delta$ and $a_0\in \mathbb{R}$
denotes the first coordinate of $z_0$, then clearly
\begin{displaymath}
\mathrm{card}(Z_{\delta}\cap B)\leq \mathrm{card}(A\cap
[a_0-\rho,a_0+\rho])\lesssim_{c_0,r} \frac{\rho}{\delta}.
\end{displaymath}
\end{remark}



\begin{cor}[Kakeya inequality for parabolas]\label{cor:LinKakeyaParabolas}
Let $r>0$, $c_0>0$ and $\delta_0>0$, and let $I\subset \mathbb{R}$
be a compact interval. Denote $\gamma(s):=(s^2/2,s,1)$. Let
$Z_{\delta} \subseteq  [-r,r]^3$ satisfy $\abs{a-a'}\geq
c_0\delta$ for all distinct $(a,b,c),(a',b',c')\in Z_{\delta}$ and
some $\delta\in (0,\delta_0)$. Set $\Gamma_z = \{(s,\langle
\gamma(s),z\rangle):\, s\in I\}$, $z\in Z_{\delta}$. Then for all
$\varepsilon>0$
\begin{equation*}
    \int_{[0,1]^2} \left(\sum_{z\in Z_{\delta}} \chi_{\Gamma_z^\delta}\right)^{3/2}
    \leq C_\varepsilon(c_0,\delta_0,r,I)\delta\cdot \delta^{-\varepsilon}\mathrm{card}(Z_{\delta}),
\end{equation*}
where $\Gamma_z^\delta$ is the Euclidean $\delta$-neighborhood of
$\Gamma_z$, and $C_\varepsilon(c_0,\delta_0,r,I)>0$ is a constant.
\end{cor}

\begin{proof}
It follows from the discussion before and after Theorem
\ref{t:LinKakeyaPYZ}, that this theorem is applicable to the
specific curve $\gamma$ in the corollary. Fix $\varepsilon>0$
arbitrarily. Let $\delta_1(\varepsilon) =
\delta_1(\varepsilon,r,I)>0$ be such that the thesis of Theorem
\ref{t:LinKakeyaPYZ} holds for all $\delta\in
(0,\delta_1(\varepsilon))$. \par By Remark \ref{r:Sep_a} and the
assumptions on $Z_{\delta}$, we know that there is a constant
$C_0=C_0(r,c_0)>0$ such that
\begin{equation*}
    \mathrm{card}(Z_{\delta}\cap B)\leq C_0 \rho/\delta,
\end{equation*}
for all balls $B\subseteq \mathbb{R}^3$ of radius
$\rho\geq\delta$. Let $\delta_\varepsilon\in
(0,\delta_1(\varepsilon)]$ be such that $C_0\leq
\delta^{-\varepsilon}$ for all $\delta\in (0,\delta_\varepsilon)$.
If $\delta\in (0,\delta_\varepsilon)$, then $Z_{\delta}$ satisfies
all the assumptions of Theorem \ref{t:LinKakeyaPYZ}, yielding the
desired inequality. If $\delta\in [\delta_\varepsilon,\delta_0)$,
then $\mathrm{card}(Z_{\delta})\lesssim \max\{
r/(c_0\delta),1\}\lesssim_{\varepsilon,c_0,r,I}1$. Hence
\begin{align*}
\int_{[0,1]^2} \left(\sum_{z\in Z_{\delta}}
\chi_{\Gamma_z^\delta}\right)^{3/2} &\leq
\mathrm{card}(Z_{\delta})^{1/2+1} \lesssim_{\varepsilon, c_0, r,
I} \mathrm{card}(Z_{\delta}) \lesssim_{\varepsilon, c_0, r, I}
\delta^{1-\varepsilon}\mathrm{card}(Z_{\delta}). \qedhere
\end{align*}
\end{proof}

\section{Linear Kakeya inequality}\label{ss:LinKakeya}
It is well known that $L^p\to L^p$-bounds for the Euclidean Kakeya
maximal operators are equivalent to estimates on the
$L^{p'}$-interaction of Euclidean $\delta$-separated
$\delta$-tubes, where $p'=p/(p-1)$. We establish a similar result
(Proposition \ref{prop:LinKakeya_equiv_KakeyaMax}) for the
Heisenberg Kakeya maximal operator $M_{\delta}$, defined in
\eqref{eq:MaxOpDef}, and then prove Theorems
\ref{t:LinKakeyaIneqIntro} and  \ref{t:MainIntro}. The main
novelty related to the Heisenberg group lies in the proof of
Theorem \ref{t:LinKakeyaIneqIntro}. The proofs of Propositions
\ref{prop:ti_tubes_implies_KakeyaMax} and
\ref{prop:LinKakeya_equiv_KakeyaMax} follow almost verbatim their
Euclidean analogs as presented in \cite{MR3617376}. The most
obvious difference  is the fact that we consider
$\delta^2$-separated (Heisenberg) $\delta$-tubes, rather than
$\delta$-separated $\delta$-tubes. This mismatch of exponents
originates in the following lemma, see also Remark
\ref{rmk:exponent_mismatch}.
\begin{lemma}\label{l:KakeyaMax_small_angle_diff}
There exists a constant $c_1\in (0,1)$ such that

\begin{equation}\label{eq:l:KakeyaMax_small_angle_diff}
M_\delta f(e)\lesssim M_{2\delta} f(e'),
\end{equation}
for all $f\in L^1_{\mathrm{loc}}(\R^3)$, $\delta\in (0,1)$ and
$e,e'\in S^1$
with $\abs{e-e'}\leq c_1\delta^2$.
\end{lemma}
\begin{proof} Throughout the proof, we denote by ``$e$'' both an element $e\in S^1$ and its embedding $(e,0)\in \mathbb{R}^2 \times \{0\}$ into $\mathbb{H}^1$. It is well known that there exists a constant $C>1$ such that \begin{equation}\label{eq:constC}
d_{\mathbb{H}}(e,e') \leq C \sqrt{|e-e'|},\quad e,e'\in S^1.
\end{equation}
This can be seen explicitly by considering arbitrary points
$e=(a,b),e'=(a',b')\in S^1$ and noting that
$\abs{ab'-ba'}=\abs{(a-a')b'-(b-b')a'}=|\langle
e-e',(b',-a')\rangle|\leq |e-e'|$  and so
\begin{equation*}
    d_\mathbb{H}(e,e')^4=\abs{e-e'}^4+16\abs{\tfrac{1}{2}[ab'-ba']}^2 \lesssim |e-e'|^2.
\end{equation*}

We set $c_1:=1/C^2$ for a fixed $C>1$ as in \eqref{eq:constC}, and
now show that we can cover a Heisenberg $\delta$-tube with another
Heisenberg $2\delta$-tube whenever their directions are $c_1
\delta^2$-close. That is, we claim that for each $e,e'\in S^1$,
$\delta>0$ and $y\in\Hei$ it holds
\begin{equation}\label{eq:tube_inclusion}
    \abs{e-e'}\leq c_1\delta^2\quad\Rightarrow\quad T^\delta(y,e)\subseteq T^{2\delta}(y,e').
\end{equation}

Let $e,e'\in S^1$ be such that $\abs{e-e'}\leq c_1 \delta^2$ and
note that $d_\mathbb{H}(e,e')\leq \delta$. Let
$y\in\Hei,T=T_\delta(y,e)$ and set $T':=T_{2\delta}(y,e')$. For
$x\in T$ there is $s_x\in \left[-\tfrac{1}{2},\tfrac{1}{2}\right]$
such that $d_\mathbb{H}(x,y\cdot s_xe)\leq\delta$. Since
$d_\mathbb{H}(y\cdot s_x e, y\cdot s_x e')
=\abs{s_x}d_\mathbb{H}(e,e')\leq \delta/2$, we have
\begin{equation*}
    \dist_\mathbb{H}{(x,y\cdot I_{e'})}\leq d_\mathbb{H}(x,y\cdot s_x e) + d_\mathbb{H}(y\cdot s_xe,y\cdot s_x e')
    \leq \delta + \delta/2
\end{equation*}
and so $T\subseteq T'$. We can finally prove
\eqref{eq:l:KakeyaMax_small_angle_diff}. If $T$ and $T'$ are as
before and $f\in L^1_\mathrm{loc}(\R^3)$, then
\begin{equation*}
    \frac{1}{\abs{T}}\int_{T}\abs{f}\leq \frac{1}{\abs{T}}\int_{T'}\abs{f}
        \lesssim M_{2\delta} f(e'),
\end{equation*}
and taking the supremum over $y\in\Hei$ yields the claim.
\end{proof}

We next consider the operator norm of the Heisenberg Kakeya
maximal operator, where $\|\cdot\|_{L^p(S^1)}$ is computed with
respect to the standard length measure $\sigma$ on $S^1$, and
$\|\cdot\|_{L^p(\mathbb{H}^1)}=\|\cdot\|_{L^p(\mathbb{R}^3)}$ with
the Lebesgue measure.
\begin{proposition}\label{prop:ti_tubes_implies_KakeyaMax}
Let $M>0$, $1<p<\infty$, $p'=p/(p-1)$ and $\delta\in(0,1)$.
Suppose that for all $t_1,\dots,t_m>0$ with
$\sum_{j=1}^mt_j^{p'}\delta^2\leq 1$ and Heisenberg $\delta$-tubes
$T_1,\dots, T_m$ in $\delta^2$-separated directions it holds
\begin{equation*}
    \norm{\sum_{j=1}^mt_j\chi_{T_j}}_{L^{p'}(\R^3)} \leq M.
\end{equation*}
Then $\norm{M_{\delta/2}}_{p\to p}\lesssim_p \delta^{-1}M$.
\end{proposition}
\begin{proof}
Let $c_1\in (0,1)$ be as in Lemma
\ref{l:KakeyaMax_small_angle_diff}. Let $\{e_1,\dots,e_m\}\subset
S^1$ be a maximal $c_1(\delta/2)^2$-separated set. Let $f\in
L^p(\R^3)$ and note that $M_{\delta/2} f(e)\lesssim M_{\delta} f
(e_j)$ for each $e\in B_E(e_j,c_1(\delta/2)^2)$. We thus have
\begin{equation*}
\norm{M_{\delta/2} f}_{L^p(S^1)}^p = \int_{S^1}M_{\delta/2}
f(e)^p\,d\sigma(e) \leq
\sum_{j=1}^m\int_{B_E(e_j,c_1(\delta/2)^2)}M_{\delta/2}
f(e)^p\,d\sigma(e) \lesssim_p \sum_{j=1}^m\delta^2M_{\delta}
f(e_j)^p.
\end{equation*}
By a dual characterisation of the $\ell^p$-norm there are
$b_1,\dots, b_m\geq 0$ such that $\sum_{j=1}^mb_j^{p'}=1$ and
$\sum_{j=1}^m M_{\delta} f (e_j)^p = \left(\sum_{j=1}^m b_j
M_{\delta} f(e_j)\right)^p$; thus
\begin{equation}\label{eq:prop:ti_tubes_implies_KakeyaMax_1}
\norm{M_{\delta/2} f}_{L^p(S^1)} \lesssim_p
\delta^{2/p}\sum_{j=1}^m b_j M_{\delta} f(e_j).
\end{equation}
Let $T_1,\dots, T_m$ be Heisenberg $\delta$-tubes having
directions respectively $e_1,\dots, e_m$, satisfying
$\Barint_{T_j}\abs{f}\geq (1/2)M_{\delta}f (e_j)$. Then
\eqref{eq:prop:ti_tubes_implies_KakeyaMax_1} gives
\begin{align*}
\norm{M_{\delta/2} f}_{L^p(S^1)} &\lesssim_p
\delta^{2/p}\delta^{-3}\sum_{j=1}^mb_j\int_{T_j}\abs{f}
= \delta^{-1}\delta^{-2/{p'}}\int\sum_{j=1}^m b_j\chi_{T_j}\abs{f} \\
&\leq
\delta^{-1}\norm{\sum_{j=1}^m\delta^{-2/{p'}}b_j\chi_{T_j}}_{L^{p'}(\R^3)}\norm{f}_{L^p(\R^3)}.
\end{align*}
Note that $T_1,\dots, T_m$ are Heisenberg $\delta$-tubes, but have
directions which are only $c_1(\delta/2)^2$-separated, while we
need them to be $\delta^2$-separated. (Recall that $c_1\in(0,1)$.)
However, one can show that there are $J_1,\dots, J_k\subseteq
\{1,\dots,m\}$ with $k\lesssim
\delta^2/\big(c_1(\delta/2)^2\big)\lesssim 1$, such that $\{e_j :
j\in J_i\}$ is $\delta^2$-separated for each $i$, and
$\cup_{i=1}^k J_i = \{1,\dots, m\}$. Since $\sum_{j\in
J_i}(\delta^{-2/{p'}}b_j)^{p'}\delta^2\leq \sum_{j=1}^m
(\delta^{-2/{p'}}b_j)^{p'}\delta^2 = 1$, we have by assumption
that
\begin{align*}
\norm{M_{\delta/2} f}_{L^p(S^1)}
&\lesssim_p \delta^{-1}\sum_{i=1}^k\norm{\sum_{j\in J_i}\delta^{-2/{p'}}b_j\chi_{T_j}}_{L^{p'}(\R^3)}\norm{f}_{L^p(\R^3)} \\
&\leq \delta^{-1}Mk\norm{f}_{L^p(\R^3)} \lesssim
\delta^{-1}M\norm{f}_{L^p(\R^3)}.\qedhere
\end{align*}
\end{proof}
\begin{proposition}\label{prop:LinKakeya_equiv_KakeyaMax}
Let $M\geq 1$, $\beta\in\R$,  $1<p<\infty$, and $p'=p/(p-1)$. Then
the following are equivalent:
\begin{itemize}
    \item for all $\varepsilon>0$, $\delta\in (0,1)$ and every family $\calT$ of Heisenberg $\delta$-tubes
        having $\delta^2$-separated direction it holds
        \begin{equation}\label{eq:LinKakeya_equiv_KakeyaMax_1}
            \norm{\sum_{T\in\calT}\chi_{T}}_{L^{p'}(\R^3)}\lesssim_{p,\beta,\varepsilon}M\delta^{\beta-\varepsilon}(\delta^2\mathrm{card}(\calT))^{1/{p'}};
        \end{equation}
    \item for all $\varepsilon>0$ and $\delta\in (0,1)$ it holds
        \begin{equation}\label{eq:LinKakeya_equiv_KakeyaMax_2}
            \norm{M_\delta}_{p\to p}\lesssim_{p,\beta,\varepsilon} M \delta^{\beta-1-\varepsilon}.
        \end{equation}
\end{itemize}
\end{proposition}
\begin{remark}
Note that for $\beta$ too large or too small inequality
\eqref{eq:LinKakeya_equiv_KakeyaMax_1} becomes respectively false
or trivial. For example, if $\beta>1/{p'}$, then
\eqref{eq:LinKakeya_equiv_KakeyaMax_1} is false, as can be seen
taking $\sim \delta^{-2}$ disjoint Heisenberg $\delta$-tubes;
while it follows from the triangle inequality if $\beta\leq
3/{p'}-2$. The same can be said about
\eqref{eq:LinKakeya_equiv_KakeyaMax_2}. We further comment on this
when we assess the sharpness of Theorem \ref{t:MainIntro} and
Theorem \ref{t:LinKakeyaIneqIntro}; see Remark
\ref{rmk:sharpness}.
\end{remark}
\begin{proof}
We start by showing that \eqref{eq:LinKakeya_equiv_KakeyaMax_1}
implies \eqref{eq:LinKakeya_equiv_KakeyaMax_2}. Note that
$\norm{M_\delta}_{p\to p}\lesssim_{p}1\lesssim_\beta
M\delta^{\beta-1-\varepsilon}$ for all $\delta\in [1/2,1)$ and
$\varepsilon>0$, so it is enough to show that
\eqref{eq:LinKakeya_equiv_KakeyaMax_2} holds for $\delta\in
(0,1/2)$ and $\varepsilon>0$. From Proposition
\ref{prop:ti_tubes_implies_KakeyaMax} it then suffices to prove
that for all $\delta\in (0,1)$, $t_1,\dots, t_m>0$ with
$\sum_{j=1}^mt_j^{p'}\delta^2\leq 1$ and $\delta^2$-separated
Heisenberg $\delta$-tubes $T_1,\dots,T_m$, it holds
\begin{equation}\label{eq:LinKakeya_equiv_KakeyaMax_3}
\norm{\sum_{j=1}^m t_j
\chi_{T_j}}_{L^{p'}(\mathbb{R}^3)}\lesssim_{p, \beta,\varepsilon}
M\delta^{\beta-\varepsilon}.
\end{equation}
Let $t_1,\dots, t_m$ and $T_1,\dots, T_m$ be as above. Suppose we
have proven \eqref{eq:LinKakeya_equiv_KakeyaMax_3} under the
additional assumption $t_j\geq \delta^{3/p+\beta-1}$ for all $j$.
Then
\begin{align*}
\norm{\sum_{j=1}^m t_j \chi_{T_j}}_{L^{p'}(\mathbb{R}^3)} &\leq
\norm{\sum_{j : t_j<\delta^{3/p+\beta-1}}
t_j\chi_{T_j}}_{L^{p'}(\mathbb{R}^3)}
+ \norm{\sum_{j : t_j\geq \delta^{3/p+\beta-1}} t_j\chi_{T_j}}_{L^{p'}(\mathbb{R}^3)} \\
&\lesssim_{p,\beta,\varepsilon}
\delta^{3/p+\beta-1}m\delta^{3/{p'}} + M\delta^{\beta-\varepsilon}
\lesssim \delta^{3/p-3+3/{p'}}\delta^\beta +
M\delta^{\beta-\varepsilon} \lesssim M\delta^{\beta-\varepsilon},
\end{align*}
where we have used $m\lesssim\delta^{-2}$ ($\delta^2$-separated
directions) and $M\delta^{-\varepsilon}\geq 1$ for all $\delta\in
(0,1)$. We can thus assume w.l.o.g.\ $\delta^{3/p+\beta-1}\leq
t_j\leq \delta^{-2/{p'}}$. (If $3/p+\beta-1<-2/{p'}$, then this
condition is vacuous and
\eqref{eq:LinKakeya_equiv_KakeyaMax_3} follows from the triangle inequality, as we have just seen.) \\
For each $k\in\Z$, let $J_k :=\{j : 2^{k-1}<t_j\leq 2^k\}$ and
note that $\mathrm{card}(\{k : J_k \neq
\varnothing\})\lesssim_{p,\beta} \log(1/\delta)+1$. Then from
\eqref{eq:LinKakeya_equiv_KakeyaMax_1} we get
\begin{equation*}
\norm{\sum_{j=1}^mt_j\chi_{T_j}}_{L^{p'}(\mathbb{R}^3)} \leq
\sum_{k\in\Z} 2^k\norm{\sum_{j\in
J_k}\chi_{T_j}}_{L^{p'}(\mathbb{R}^3)}
\lesssim_{p,\beta,\varepsilon}
M\delta^{\beta-\varepsilon/2}\delta^{2/{p'}}\sum_{k\in\Z}2^k(\mathrm{card}(J_k))^{1/{p'}}.
\end{equation*}
Since $2^{kp'}\mathrm{card}(J_k)\lesssim_p \sum_{j=1}^m
t_j^{p'}\leq \delta^{-2}$, it follows that
\begin{equation*}
    \sum_{k\in\Z}2^k(\mathrm{card}(J_k))^{1/{p'}}
    \lesssim_p \delta^{-2/{p'}}\mathrm{card}(\{k : J_k\neq \varnothing\})
    \lesssim_{p,\beta} \delta^{-2/{p'}}(\log(1/\delta)+1).
\end{equation*}
Finally, we have
\begin{equation*}
\norm{\sum_{j=1}^mt_j\chi_{T_j}}_{L^{p'}(\mathbb{R}^3)}
\lesssim_{p,\beta,\varepsilon}
M\delta^{\beta-\varepsilon/2}\delta^{2/{p'}}
\delta^{-2/{p'}}(\log(1/\delta)+1)
\lesssim_{\varepsilon}M\delta^{\beta-\varepsilon},
\end{equation*}
which concludes the proof of the first part of the statement. We
now show that \eqref{eq:LinKakeya_equiv_KakeyaMax_2} implies
\eqref{eq:LinKakeya_equiv_KakeyaMax_1}. We are assuming that
\eqref{eq:LinKakeya_equiv_KakeyaMax_2} holds for all $\delta\in
(0,1)$, but adjusting the implicit constant we see that it
actually holds also for $\delta\in (0,2)$. (We made a similar
remark at the beginning of this proof.) Let $\delta\in (0,1)$ and
$\calT$ be a family of Heisenberg $\delta$-tubes having
$\delta^2$-separated directions and let $e_T\in S^1$ denote the
direction of $T$, for each $T\in\calT$. Let $g\in L^p(\R^3)$,
$g\geq 0$, with $\norm{g}_{L^p(\mathbb{R}^3)}\leq 1$. Set $c_2 :=
\min\{c_1,1/2\}$, where $c_1\in (0,1)$ is as in Lemma
\ref{l:KakeyaMax_small_angle_diff}. Then
\begin{align*}
\int g\sum_{T\in\calT}\chi_T &\sim
\sum_{T\in\calT}\delta^3\Barint_Tg \leq
\sum_{T\in\calT}\delta^3M_\delta g (e_T)
\lesssim \delta^3\sum_{T\in\calT} \Barint_{B_E(e_T,c_2\delta^2)}M_{2\delta}g\,d\sigma \\
&\lesssim \delta\, \left[\sigma\left(\bigcup_{T\in\calT}
B_E(e_T,c_2\delta^2)\right)\right]^{1/{p'}}\norm{M_{2\delta}g}_{L^{p}(S^1)}
\lesssim_{p,\beta,\varepsilon}
M\delta^{\beta-\varepsilon}(\delta^2\mathrm{card}(\calT))^{1/{p'}}.
\end{align*}
Taking the supremum over the $g$ as above, we obtain
\eqref{eq:LinKakeya_equiv_KakeyaMax_1}.
\end{proof}
We are now ready to prove Theorem \ref{t:LinKakeyaIneqIntro},
which we restate. \LinKakeyaIneqIntro*
\begin{proof}[Proof of Theorem \ref{t:LinKakeyaIneqIntro}]
Let $S_j :=
\{(\cos\varphi,\sin\varphi):\abs{\varphi-j\pi/2}\leq\pi/4\}$ for
$j\in\{0,1,2,3\}$, and let $\calT_j\subseteq\calT$ be the set of
Heisenberg tubes in $\calT$ with direction in $S_j$. It is enough
to show that \eqref{eq:t:LinKakeyaIneq} holds in the case $\calT
=\calT_1$ or, equivalently, for $\calT=\calT_j$ for some $j$.
(Recall that $(x',x_3)\in\R^2\times\R \mapsto (Ox',x_3)$ is a
$(d_\mathbb{H},d_\mathbb{H})$-isometry for each $O\in SO(2)$.)
Indeed, we then have
\begin{equation*}
\norm{\sum_{T\in\calT}\chi_T}_{L^{3/2}(\mathbb{R}^3)}\leq\sum_{j=0}^3\norm{\sum_{T\in\calT_j}\chi_T}_{L^{3/2}(\mathbb{R}^3)}
\lesssim_\varepsilon\sum_{j=0}^3
(\delta^{3-\varepsilon}\mathrm{card}(\calT_j))^{2/3} \lesssim
(\delta^{3-\varepsilon}\mathrm{card}(\calT))^{2/3}.
\end{equation*}
We can now assume $\calT = \calT_1$. Let $\calB$ be a collection
of Heisenberg balls of radius $1/2$ covering $\R^3$ and such that
the balls with same centers and four times larger radius have
absolutely bounded overlaps:
$$
\sup_{x\in \mathbb{H}^1}\mathrm{card}\left(\{B\in \mathcal{B}:\,
x\in 4B\}\right)\lesssim 1.
$$
If \eqref{eq:t:LinKakeyaIneq} holds with the left-hand side having
integration domain an arbitrary $B\in\calB$, we then have
\begin{align*}
\int\left(\sum_{T\in\calT}\chi_T\right)^{3/2}
&\sim\sum_{B\in\calB}
    \int_B\left(\sum_{T\in\calT : T\cap B \neq \varnothing}\chi_T\right)^{3/2}
    \lesssim_\varepsilon \sum_{B\in\calB} \delta^{3-\varepsilon}
        \mathrm{card}\left(\{T\in\calT : T\cap B \neq \varnothing\}\right) \\
&=\delta^{3-\varepsilon}\sum_{T\in\calT}\mathrm{card}\left(\{B\in\calB:T\cap
B \neq\varnothing\}\right) \lesssim
\delta^{3-\varepsilon}\mathrm{card}(\calT),
\end{align*}
as $\mathrm{card}\left(\{B\in\calB : T\cap
B\neq\varnothing\}\right)$ is bounded by an absolute constant
since $T_{\delta}(y,e)\cap B \neq \varnothing$ implies that $y\in
4B$.

We can then focus on a single ball $B\in\calB$ and assume that
$T\cap B \neq\varnothing$ for all $T\in\calT$. Moreover, by left
translation, we can  assume $B$ to be centred at
$\left(0,\tfrac{1}{2},\tfrac{1}{2}\right)$; this ensures that
$\pi_\mathbb{W}(B)\subseteq [0,1]^2$, where we have identified
$\mathbb{W}=\{0\}\times\R^2$ with $\R^2$. \par By Lemma
\ref{l:tube_parab} we know that for each $T\in\calT$ there is a
Euclidean $\sim \delta^2$-neighborhood $P_T$ of a parabola arc
such that $\pi_\mathbb{W}(T)\subseteq P_T$ and so $\chi_T(w\cdot
l)\leq \chi_{P_T}(w)$ for all $w\in\mathbb{W}$ and
$l\in\mathbb{L}=\mathbb{R}\times \{(0,0)\}$. Also, the inclusion
$\pi_\mathbb{W}(B)\subseteq [0,1]^2$ gives $\chi_B(w\cdot l)\leq
\chi_{[0,1]^2}(w)$. Thus
\begin{align}\label{eq:t:LinKakeyaIneq_1}
\int_B \left(\sum_{T\in\calT}\chi_T\right)^{3/2}
&\overset{\eqref{eq:Fubini}}{=}
\int_{\mathbb{W}}\int_{\mathbb{L}}\left(\sum_{T\in\calT}\chi_T(w\cdot
l)\right)^{1/2+1}\chi_B(w\cdot l)
    \, dl\,dw \nonumber \\
&\leq
\int_{[0,1]^2}\left(\sum_{T\in\calT}\chi_{P_T}(w)\right)^{1/2}
    \sum_{T\in\calT}\calH^1\big(T\cap \pi_{\mathbb{W}}^{-1}(\{w\})\big)\,dw \nonumber \\
&\overset{\mathrm{Lem.}\ref{l:tube_cap_line_lesssim_delta}}{\lesssim}
    \delta\,\int_{[0,1]^2}\left(\sum_{T\in\calT}\chi_{P_T}(w)\right)^{3/2}\,dw,
\end{align}
where in the last line we have also used the fact that $T\cap
\pi_{\mathbb{W}}^{-1}(\{w\}) = \varnothing$ if $w\notin P_T$ and
so $\calH^1\big(T\cap \pi_{\mathbb{W}}^{-1}(\{w\})\big)\lesssim
\delta\chi_{P_T}(w)$. \par To conclude the proof, we show that we
can apply Corollary \ref{cor:LinKakeyaParabolas} to the parabola
neighborhoods $\{P_T : T\in\calT\}$. To do so, we employ Lemma
\ref{l:tube_parab} and the formulae of \eqref{eq:(a,b,c)}. \par
For each $T\in\calT$ there are a compact interval $I$ and
$(a,b,c)\in\R^3$ (both depending on $T$) such that $P_T$ is a
neighborhood of the parabola arc $\gamma_{(a,b,c)}(I)$. Since
$\calT=\calT_1$ we know that $a\in [-1,1]$ and so there is an
absolute constant $C_0>0$ such that $P_T$ is contained in the
$C_0\delta^2$-neighborhood of its ``core curve''
$\gamma_{(a,b,c)}(I)$. (Recall that $P_T$ is the $\sim
(1+\abs{a})\delta^2\sim\delta^2$-neighborhood of
$\gamma_{(a,b,c)}(I)$.) Also, $T\cap B\neq \varnothing$ and
$\delta\in (0,1)$ imply that there is an absolute constant $r\geq
1$ such that $b,c\in [-r,r]$ and $I\subseteq [-r,r]$. \par Let
$Z\subseteq [-r,r]^3$ be the set $(a,b,c)\in\R^3$ corresponding to
$\calT$. Note that $a\in [-1,1]\mapsto \frac{(a,1)}{\abs{(a,1)}}$
is Lipschitz continuous; since $\calT$ is a family of tubes in
$\delta^2$-separated directions, it follows that there is an
absolute constant $c_0>0$ such that $\abs{a-a'}\geq c_0
(C_0\delta^2)$ for all distinct $(a,b,c),(a',b',c')\in Z$. Thus,
Corollary \ref{cor:LinKakeyaParabolas} applies to $Z$ with
$C_0\delta^2$ in place of $\delta$. Taking $\delta_0=C_0$ in
Corollary \ref{cor:LinKakeyaParabolas}, its conclusion holds for
$C_0\delta^2\in (0,C_0)$. Hence, for $\delta\in (0,1)$ and each
$\varepsilon>0$,
\begin{equation}\label{eq:t:LinKakeyaIneq_2}
\int_{[0,1]^2}\left(\sum_{T\in\calT}\chi_{P_T}\right)^{3/2} \leq
\int_{[0,1]^2}\left(\sum_{z\in
Z}\chi_{\Gamma^{C_0\delta^2}_z}\right)^{3/2} \leq
C_\varepsilon\delta^{2-\varepsilon}\card{(Z)},
\end{equation}
where $\Gamma^{C_0\delta^2}_z$ is the $C_0\delta^2$-neighborhood
the parabola arc $\Gamma_z = \gamma_z([-r,r])$. Since
$\mathrm{card}(Z)=\mathrm{card}(\calT)$,
\eqref{eq:t:LinKakeyaIneq_2} and \eqref{eq:t:LinKakeyaIneq_1}
conclude the proof.
\end{proof}
We can finally prove Theorem \ref{t:MainIntro}. \MainIntro*
\begin{proof}[Proof of Theorem \ref{t:MainIntro}]
Theorem \ref{t:LinKakeyaIneqIntro} shows that
\eqref{eq:LinKakeya_equiv_KakeyaMax_1} holds with $p'=3/2$, $M=1$
and $\beta=2/3$. Thus, Proposition
\ref{prop:LinKakeya_equiv_KakeyaMax} immediately gives
$\norm{M_\delta}_{3\to 3}\lesssim_\varepsilon
\delta^{2/3-1-\varepsilon} = \delta^{-1/3-\varepsilon}$ for all
$\varepsilon>0$ and $\delta\in (0,1)$.
\end{proof}

The exponent $p'=3/2$ in Theorem \ref{t:LinKakeyaIneqIntro} is the
best possible as explained in Remark \ref{rmk:sharpness}.

\begin{remark}\label{rmk:exponent_mismatch}
As mentioned in the introduction to this section, a peculiarity of
these results is that they involve $\delta^2$-separated Heisenberg
$\delta$-tubes, rather than $\delta$-separated ones. There are
several reasons for this:
\begin{enumerate}
    \item
         Lemma \ref{l:KakeyaMax_small_angle_diff}
        is based on the implication \eqref{eq:tube_inclusion},
        which does not hold if $e,e'$ are only $\sim\delta$-close.
        To see this, consider $e=\frac{1}{\sqrt{1+\delta^2}}(1,\delta)$ and $e'=(1,0)$.
        Then $|e-e'|\sim \delta$ and $x=(\tfrac{1}{2\sqrt{1+\delta^2}},\tfrac{\delta}{2\sqrt{1+\delta^2}},0)\in T_{\delta}(0,e)$,
        yet $x\notin T_{2\delta}(0,e')$ if $\delta$ is small enough, since
        \begin{displaymath}
            d_{\mathbb{H}}(x,(se,0))\gtrsim \left|\tfrac{1}{2\sqrt{1+\delta^2}}-s\right|+
            \sqrt{\tfrac{\delta |s|}{2\sqrt{1+\delta^2}}} \gg \delta
        \end{displaymath}
        for all sufficiently small $\delta>0$ (uniformly in $s\in[-1/2,1/2]$).
    \item Heisenberg $\delta$-tubes (roughly speaking) project to $\sim\delta^2$-neighborhoods of parabolas
        (Lemma \ref{l:tube_parab}) while the $\delta$-separation is preserved.
        In order to apply Corollary \ref{cor:LinKakeyaParabolas}, we need the separation to be at least of the same order as the radius of the neighborhoods.
    \item Theorem \ref{t:LinKakeyaIneqIntro} applies to the sparser $\delta$-separated $\delta$-tubes.
        However, the usual argument would not yield the sharp lower bound on the Minkowski dimension if we only considered $\delta$-separated
        $\delta$-tubes in Theorem \ref{t:LinKakeyaIneqIntro}. See Remark \ref{rmk:Dimension_conclusionMinkowski}.

\end{enumerate}
\end{remark}

\section{Conclusion about Hausdorff dimension}\label{ss:Dim}

We recall the following definition from \cite{MR4439466}.
\begin{definition}\label{d:HeisKakeya}
We say that $E\subseteq \Hei$ is a \emph{Heisenberg Kakeya set} 
if for every $e\in S^1$ there is a $y\in\Hei$ such that $y\cdot
I_e\subseteq E$, where $I_e = \{(se,0):s\in[-1/2,1/2]\}$ is a
horizontal unit line segment.
\end{definition}
Since the projection  $\pi:(x_1,x_2,x_3)\mapsto (x_1,x_2)$ maps a
Heisenberg Kakeya set $E$ onto a Kakeya set in $\mathbb{R}^2$, it
follows from the validity of the Kakeya conjecture in
$\mathbb{R}^2$ that $E$ must have Euclidean Hausdorff dimension at
least $2$. However, since $\pi$ can increase Hausdorff dimension
in the metric $d_{\mathbb{H}}$, the dimension of $E$ with respect
to this metric has  to be studied separately. Liu showed in
\cite{MR4439466} that Heisenberg Kakeya sets in $\mathbb{H}^1$
have Hausdorff dimension at least $3$ with respect to
$d_{\mathbb{H}}$. Theorem \ref{t:MainIntro} provides an
alternative proof of this result (Proposition \ref{p:Dimension
conclusion}). For illustration, we first explain how the linear
Heisenberg Kakeya inequality can be used to prove the sharp bound
for the lower \emph{Minkowski (box-counting) dimension} of
Heisenberg Kakeya sets in $\mathbb{H}^1$ (Remark
\ref{rmk:Dimension_conclusionMinkowski}). This is a weaker
statement since the lower Minkowski dimension of a set is always
greater than its Hausdorff dimension. Nevertheless, it provides a
good heuristic for the relation between the numerology of Theorem
\ref{t:LinKakeyaIneqIntro} and the dimension of Heisenberg Kakeya
sets (see also Remark \ref{rmk:exponent_mismatch}).

\begin{remark}\label{rmk:Dimension_conclusionMinkowski}
Let $E^\delta$ be the $\delta$-enlargement w.r.t.\ $d_\mathbb{H}$
of a Heisenberg Kakeya set $E$. We show that
$\abs{E^\delta}\gtrsim_\varepsilon \delta^{1+\varepsilon}=
\delta^{4-3+\varepsilon}$ for all $\varepsilon>0$ and $\delta\in
(0,1)$; this implies by a standard argument that $E$ has lower
Minkowski dimension (w.r.t.\ $d_\mathbb{H}$) at least $3$. Fix a
$\delta^2$-separated set of directions in $S^1$, and for any one
of them consider the unit segment contained in $E$ having such
direction. The Heisenberg $\delta$-neighborhood of any such
segment is a Heisenberg $\delta$-tube which is contained in
$E^\delta$. Let $\calT$ denote the collection of such tubes; it is
a $\delta^{2}$-separated family of Heisenberg $\delta$-tubes.
Hence
\begin{equation*}
\big\|\sum_{T\in \mathcal{T}}\chi_T \big\|_{L^1(\mathbb{R}^3)}
\leq \left(|\bigcup_{T\in \mathcal{T}} T|\right)^{\frac{1}{3}}\,
\big\|\sum_{T\in
\mathcal{T}}\chi_T\big\|_{L^{3/2}(\mathbb{R}^3)}
\overset{\text{Thm.}
\ref{t:LinKakeyaIneqIntro}}{\lesssim_\varepsilon}
\delta^{-\varepsilon} \left(|\bigcup_{T\in \mathcal{T}}
T|\right)^{\frac{1}{3}}\, \big\|\sum_{T\in
\mathcal{T}}\chi_T\big\|_{L^{1}(\mathbb{R}^3)}^{\frac{2}{3}},
\end{equation*}
which finally implies
\begin{equation}\label{eq:Dimension_conclusionMinkowski}
\abs{E^\delta}\geq |\bigcup_{T\in \mathcal{T}} T|
\gtrsim_\varepsilon \delta^\varepsilon \big\|\sum_{T\in
\mathcal{T}}\chi_T\big\|_{L^{1}(\mathbb{R}^3)} \sim
\delta^{3+\varepsilon}\card{(\calT)}\sim \delta^{1+\varepsilon}.
\end{equation}
We stress that had we proven Theorem \ref{t:LinKakeyaIneqIntro}
only for $\delta$-separated $\delta$-tubes, we would have the
weaker $\abs{E^\delta}\gtrsim_\varepsilon
\delta^{2+\varepsilon}=\delta^{4-2+\varepsilon}$ in
\eqref{eq:Dimension_conclusionMinkowski}. This would only give the
non-sharp lower bound $\underline{\dim}_M E\geq 2$.
\end{remark}

We reprove Liu's result as a corollary of Theorem
\ref{t:MainIntro}:

\begin{proposition}\label{p:Dimension conclusion}
If $E\subseteq \mathbb{H}^1$ is a Kakeya set of horizontal unit
line segments, then $\dim_H E \geq 3$.
\end{proposition}

\begin{proof} With Theorem \ref{t:MainIntro} at hand, this
follows by a standard argument (found e.g.\ in \cite{MR2003254} or
\cite{MR3617376}). Let $E\subseteq \mathbb{H}^1$ be a Heisenberg
Kakeya set and fix $0<\alpha<3$. Let $B_j=B(x_j,r_j)$ be an
arbitrary cover of $E$ by Kor\'{a}nyi balls of radius $r_j \leq
1$. It suffices to show that $\sum_{j} r_j^{\alpha}\gtrsim 1$. For
each $e\in S^1$, we let $\ell_e={x_e}\cdot I_e$ denote a line
segment contained in $E$. As usual, we group the covering sets
into families of balls of comparable size by defining
\begin{displaymath}
J_k:= \{j:\, 2^{-k}\leq r_j < 2^{1-k}\},\quad k\in \mathbb{N}^+,
\end{displaymath}
and the set of directions in which a line segment is well covered
by balls of comparable radius,
\begin{displaymath}
S_k:=\left\{e:\, \mathcal{H}^1(\ell_e \cap \bigcup_{j\in
J_k}B_j)\geq \frac{1}{2k^2}\right\}.
\end{displaymath}
The sets $S_k$, $k\in \mathbb{N}^+$, cover all relevant
directions. If there was $e\in S^1 \setminus \bigcup_{k\in
\mathbb{N}^+} S_k$, then
\begin{displaymath}
1 = \mathcal{H}^1(E\cap \ell_e) \leq \sum_{k=1}^{\infty}
\frac{1}{2k^2}<1,
\end{displaymath}
which is impossible.

We will apply the Heisenberg Kakeya maximal operator to the
function
\begin{displaymath}
f=\chi_{F_k},\quad \text{where }F_k=\bigcup_{j\in J_k}
B(x_j,2r_j).
\end{displaymath}
However, we firstly need to show that for $e\in S_k$
\begin{equation}\label{eq:Heisenberg_need_Fubini}
    \abs{T_{2^{-k}}(x_e,e)\cap F_k}\gtrsim \frac{1}{k^2}\abs{T_{2^{-k}}(x_e,e)}.
\end{equation}
Fix $e\in S_k$ and set
    $I_k(e) := \bigcup_{j\in J_k} B_j\cap \ell_e$.
Note that for each $z\in I_k(e)$ there is a $j\in J_k$ such that 
$B(z,2^{-k})\subseteq B(x_j,2r_j)$. Thus,
\begin{equation}\label{eq:Heisenberg_no_need_Fubini_1}
    \bigcup_{z\in I_k(e)} B(z,2^{-k}) \subseteq T_{2^{-k}}(x_e,e)\cap\bigcup_{j\in J_k}B(x_j,2r_j) = T_{2^{-k}}(x_e,e)\cap F_k.
\end{equation}
We now estimate $\abs{T_{2^{-k}}(x_e,e)\cap F_k}$ in terms of
$\calH^1(I_k(e))$. Let $P\subseteq I_k(e)$ be a maximal
$2^{-(k-1)}$-separated set in $I_k(e)$ (with respect to
$d_\mathbb{H}$), then $\{B(y,2^{-k}) : y\in P\}$ is a pairwise
disjoint family of sets and so
\begin{equation}\label{eq:Heisenberg_no_need_Fubini_2}
\abs{\bigcup_{z\in I_k(e)} B(z,2^{-k})} \geq \abs{\bigcup_{y\in P}
B(y,2^{-k})} \sim(2^{-k})^4\mathrm{card}(P).
\end{equation}
On the other hand, $\{B(y,2^{-(k-1)})\cap (x_e\cdot \spa{(e)}) :
y\in P\}$ covers $I_k(e)$ and so
\begin{equation}\label{eq:Heisenberg_no_need_Fubini_3}
    2^{-(k-1)} \mathrm{card}(P)\gtrsim \calH^1(I_k(e)).
\end{equation}
Combining the inclusion \eqref{eq:Heisenberg_no_need_Fubini_1} and
the inequalities \eqref{eq:Heisenberg_no_need_Fubini_2},
\eqref{eq:Heisenberg_no_need_Fubini_3}, we have
\begin{equation*}
    \abs{T_{2^{-k}}(x_e,e)\cap F_k} \gtrsim (2^{-k})^4\mathrm{card}(P) \gtrsim (2^{-k})^3 \calH^1(I_k(e))
    \gtrsim \frac{1}{k^2}\abs{T_{2^{-k}}(x_e,e)},
\end{equation*}
as desired. This then implies that $M_{2^{-k}}f\gtrsim k^{-2}$ on
$S_k$ and so
\begin{equation}\label{eq:Max1}
    \norm{M_{2^{-k}}f}_{L^3}\gtrsim k^{-2}\sigma(S_k)^{1/3}.
\end{equation}
On the other hand, Theorem \ref{t:MainIntro} (for suitably chosen
$\varepsilon>0$ to be determined) implies that
\begin{equation}\label{eq:Max2}
\|M_{2^{-k}} f\|_{L^3} \leq C_{\varepsilon}\,
2^{k\varepsilon}\,2^{k/3} \|f\|_{L^3(\mathbb{R}^3)}
\lesssim_{\varepsilon} 2^{k\varepsilon}\,2^{k/3}
\left(\mathrm{card}(J_k)\,2^{(1-k)4}\right)^{1/3}\lesssim_{\varepsilon}
2^{k(\varepsilon-1)} \mathrm{card}(J_k)^{1/3}.
\end{equation}
Combining \eqref{eq:Max1} and \eqref{eq:Max2} yields
\begin{displaymath}
\sigma(S_k) \lesssim k^6\,
2^{3k(\varepsilon-1)}\,\mathrm{card}(J_k) \lesssim_\varepsilon
2^{3k(2\varepsilon-1)}\,\mathrm{card}(J_k).
\end{displaymath}
This shows that
\begin{displaymath}
\sum_j r_j^{3-6\varepsilon}\gtrsim \sum_{k=1}^{\infty}
\mathrm{card}(J_k) 2^{3k(2\varepsilon-1)} \gtrsim_{\varepsilon}
\sum_{k=1}^{\infty} \sigma(S_k)\gtrsim 1.
\end{displaymath}
Hence, if $3-6\varepsilon>\alpha$, we get the desired lower bound
$\sum_j r_j^{\alpha} \gtrsim 1$.
\end{proof}
\begin{remark}\label{rmk:sharpness}
Slight changes to the proof of Proposition \ref{p:Dimension
conclusion} show that $\norm{M_\delta}_{p\to
p}\lesssim_{\varepsilon}\delta^{-\beta-\varepsilon}$ for all
$\delta\in (0,1)$ and $\varepsilon>0$ implies $\mathrm{dim}_H E
\geq 4-p\beta$ for all Heisenberg Kakeya sets $E$. Since
$\mathrm{dim}_H E\geq 3$ is sharp, the strongest bound of this
form  we can have for $M_\delta$ is $\norm{M_\delta}_{p\to
p}\lesssim_\varepsilon \delta^{-1/p-\varepsilon}$. By Proposition
\ref{prop:LinKakeya_equiv_KakeyaMax}, this is equivalent to
\begin{equation}\label{eq:rmk:sharpness_1}
    \norm{\sum_{T\in\calT}\chi_T}_{p'}\lesssim_\varepsilon \delta^{1/p'-\varepsilon}(\delta^2\card{(\calT)})^{1/p'}
\end{equation}
for all collections $\calT$ of Heisenberg $\delta$-tubes pointing
in $\delta^2$-separated directions. But if $\calT$ is a collection
of $\sim \delta^{-2}$ Heisenberg $\delta$-tubes centred at the
origin, we have
\begin{equation*}
    \norm{\sum_{T\in\calT}\chi_T}_{p'}\gtrsim_p \delta^{-2}\delta^{4/p'}.
\end{equation*}
Thus \eqref{eq:rmk:sharpness_1} is false unless $p\geq 3$ (or,
equivalently, $p'\leq 3/2$). This shows that the integrability
exponent $p=3$ and the power of $\delta$ in Theorem
\ref{t:MainIntro} are the smallest possible, i.e.\ it is the
strongest $L^p\to L^p$-bound possible (except possibly for the
$C_\varepsilon\delta^{-\varepsilon}$ term). By Proposition
\ref{prop:LinKakeya_equiv_KakeyaMax}, also Theorem
\ref{t:LinKakeyaIneqIntro} is sharp in an analogous sense.
\end{remark}
\begin{remark}\label{rmk:interpolation}
As in the Euclidean case, Theorem \ref{t:MainIntro} and the
trivial bound $\norm{M_\delta}_{\infty\to\infty}\leq 1$ imply by
(real) interpolation (see e.g.\ \cite[Theorem 2.13]{MR3617376} or
\cite[Corollary 1.4.22]{MR3243734}) that
\begin{equation*}
    \norm{M_\delta}_{p\to p}\lesssim_{p,\varepsilon}\delta^{-1/p-\varepsilon}, \quad  p\in [3,\infty].
\end{equation*}
\end{remark}
\begin{remark}
With the same exponent of $\delta$, Theorem
\ref{t:LinKakeyaIneqIntro} with the exponent $3/2$ replaced by any
exponent strictly larger than $1$ would still yield the sharp
lower bound of Proposition \ref{p:Dimension conclusion} (and, of
course, of Remark \ref{rmk:Dimension_conclusionMinkowski}).

\end{remark}
\bibliographystyle{plain}
\bibliography{references}

\end{document}